\documentclass[a4paper,12pt]{article}

\usepackage{amsmath}
\usepackage{amsthm}
\usepackage{amssymb}

\usepackage{latexsym}
\usepackage{array}
\usepackage{enumerate}

\usepackage{graphicx}

\numberwithin{equation}{section}

\newcommand{\rhom}{R{\mathcal{H}}om}
\newcommand{\BDC}{{\mathbf{D}}^{\mathrm{b}}}

\newcommand{\shom}{{\mathcal{H}}om}

\newcommand{\CC}{\mathbb{C}}

\newcommand{\RR}{\mathbb{R}}
\newcommand{\QQ}{\mathbb{Q}}
\newcommand{\ZZ}{\mathbb{Z}}
\newcommand{\D}{\mathcal{D}}

\newcommand{\LL}{\mathcal{L}}
\newcommand{\M}{\mathcal{M}}

\newcommand{\sho}{\mathcal{O}}

\newcommand{\CH}{\mathcal{H}}

\newcommand{\Lin}{{\rm Lin}}

\newcommand{\rd}{{\rm rd}}
\newcommand{\an}{{\rm an}}

\renewcommand{\dim}{{\rm dim}}

\newcommand{\Vol}{{\rm Vol}}

\newcommand{\e}{\varepsilon}

\newcommand{\Aff}{{\rm Aff}}

\newcommand{\Int}{{\rm Int}}

\newcommand{\tl}[1]{\widetilde{#1}}

\newcommand{\simto}{\overset{\sim}{\longrightarrow}}

\newfont{\bg}{cmr10 scaled\magstep4}

\newcommand{\ubigzero}{\hbox{\bg 0}}

\newtheorem{definition}{Definition}[section]
\newtheorem{theorem}[definition]{Theorem}

\newtheorem{lemma}[definition]{Lemma}
\newtheorem{corollary}[definition]{Corollary}

\newtheorem{remark}[definition]{Remark}

\title{Monodromies at infinity of 
confluent $A$-hypergeometric functions 
}

\author{Kana ANDO 
\footnote{Institute of 
Mathematics, University  of 
Tsukuba, 1-1-1, Tennodai, 
Tsukuba, Ibaraki, 305-8571, Japan. 
E-mail: ando@graduate.chiba-u.jp}, 
Alexander ESTEROV 
\footnote{ National Research University 
Higher School of Economics \newline Faculty 
of Mathematics NRU HSE, 7 Vavilova 117312 
Moscow, Russia. 
E-mail: aesterov@hse.ru. \newline 
This study (research grant No 14-01-0152) is supported by The National Research University--Higher School of Economics' Academic Fund Program in 2014/2015. Partially supported by RFBR 
grants 13-01-00755 and 12-01-31233, MESRF grant MK-6223.2012.1, 
and the Dynasty Foundation fellowship. } 
and Kiyoshi TAKEUCHI 
\footnote{Institute of Mathematics, University  of 
Tsukuba, 1-1-1, Tennodai, 
Tsukuba, Ibaraki, 305-8571, Japan. 
E-mail: takemicro@nifty.com } }

\date{}

\begin{document}

\maketitle

\begin{abstract}
We study the monodromies at infinity of 
confluent $A$-hypergeometric functions 
introduced by Adolphson \cite{A}. In particular, 
we extend the result of \cite{T-1} for non-confluent 
$A$-hypergeometric functions to 
the confluent case. The integral representation 
by rapid decay homology cycles 
proved in \cite{E-T-2} will play a central 
role in the proof. 
\end{abstract}

\section{Introduction}\label{sec:1}

The theory of $A$-hypergeometric systems 
introduced by Gelfand-Kapranov-Zelevinsky 
\cite{G-K-Z-1} is a vast generalization 
of that of classical hypergeometric differential 
equations. We call their holomorphic solutions 
$A$-hypergeometric functions. As in the classical case, 
$A$-hypergeometric functions 
admit $\Gamma$-series expansions 
(see \cite{G-K-Z-1}, \cite{S-S-T} etc.) and integral 
representations (\cite{G-K-Z-2}). Moreover 
they have deep connections with 
many other fields of mathematics, such as 
toric varieties, projective duality, 
period integrals, mirror symmetry, 
commutative algebra, 
enumerative algebraic geometry and 
combinatorics. Also from the viewpoint of 
the $\D$-module theory (see  
\cite{H-T-T} etc.), 
$A$-hypergeometric $\D$-modules were very 
elegantly constructed 
in \cite{G-K-Z-2}. For the recent 
development of this theory see \cite{S-S-T}, 
\cite{S-W} etc. In \cite{B}, \cite{B-H}, \cite{G-K-Z-2}, 
\cite{Horja}, \cite{T-1} etc. the 
monodromies of their $A$-hypergeometric functions 
were studied. In particular, in \cite{T-1} 
the third author obtained a formula 
for their monodromies at infinity. 
The aim of this paper is to generalize it 
to the confluent $A$-hypergeometric functions 
introduced by Adolphson \cite{A}. 
In the confluent case, by the lack of
the integral representation, almost nothing 
was known about the global property of the 
confluent $A$-hypergeometric functions 
before \cite{E-T-2}. Recently in \cite{E-T-2} 
we established their integral representation  
by using Hien's theory of rapid decay 
homology groups invented 
in \cite{H-1} and \cite{H-2} 
and obtained a formula for their  
asymptotic expansion at infinity 
(see also 
Hien-Roucairol \cite{H-R}, 
Saito \cite{S1} and 
Schulze-Walther \cite{SW1}, \cite{SW2} 
for related results). We will use 
this development to prove our main theorem. 

\medskip 
In order to introduce our result, first we 
recall the definition 
of Adolphson's confluent $A$-hypergeometric systems 
in \cite{A}. In this paper, 
we essentially follow the terminology of \cite{H-T-T}. 
Let $A=\{ a(1), a(2), \ldots , 
a(N)\} \subset \ZZ^{n}$ be a finite subset 
of the lattice $\ZZ^{n}$. 
As in \cite{G-K-Z-1} and \cite{G-K-Z-2} 
assume that $A$ generates $\ZZ^{n}$. 
We denote by $\Delta$ the convex 
hull ${\rm conv} (A \cup \{ 0\})$ 
of $A \cup \{ 0\}$ in $\RR^{n}$. 
By our assumption $\Delta$ is an $n$-dimensional 
polytope. Let $c=(c_1, \ldots, c_n) \in 
\CC^{n}$ be a parameter vector. 
We identify the set $A$ with the 
$n \times N$ integer matrix
\begin{equation}
A:=\begin{pmatrix}
 {^t a(1)} & {^t a(2)} & \cdots & 
{^t a(N)} \end{pmatrix}=
(a_{i,j})_{1 \leq i \leq n, 1 \leq j \leq N}
 \in M(n, N, \ZZ)
\end{equation}
whose $j$-th column is ${^t a(j)}$.
Then Adolphson's confluent $A$-hypergeometric 
system on $X= \CC^A=\CC_z^{N}$ 
associated with the parameter vector 
$c=(c_1, \ldots, c_n) \in \CC^{n}$ is 
\begin{gather}
\left(\sum_{j=1}^{N} a_{i,j} 
z_j\frac{\partial}{\partial 
z_j}+c_i \right)
u(z)=0 \hspace{5mm} (1 \leq i \leq n), \\
\left\{ \prod_{\mu_j>0} \left(\frac{\partial}{\partial 
z_j}\right)^{\mu_j} -\prod_{\mu_j<0} 
\left(\frac{\partial}{\partial 
z_j}\right)^{-\mu_j} \right\} u(z)=0 
\hspace{5mm} (\mu \in {\rm Ker} A \cap 
\ZZ^{N}). 
\end{gather} 
This system was introduced first 
by Gelfand-Kapranov-Zelevinsky 
\cite{G-K-Z-1} when there exists 
a linear functional $l: \RR^n \longrightarrow \RR$ 
such that $l( \ZZ^n)= \ZZ$ and $A \subset l^{-1}(1)$. 
In such a case, Hotta \cite{Hotta} proved that 
it is regular holonomic i.e. non-confluent. 
This is the reason why we call the above 
Adolphson's generalization a confluent 
$A$-hypergeometric system. 
Let $D(X)$ be the Weyl algebra over $X$ 
and consider the differential operators 
\begin{eqnarray}
Z_{i,c} & := & \sum_{j=1}^{N} a_{ij} z_j
\frac{\partial}{\partial z_j}+c_i 
\hspace{5mm}(1\leq i\leq n),\\
\square_{\mu} 
& := & \prod_{\mu_j>0}\left(\frac{\partial}{\partial 
z_j}\right)^{\mu_j} -\prod_{\mu_j<0} \left(
\frac{\partial}{\partial 
z_j}\right)^{-\mu_j}\hspace{5mm} ( \mu 
\in {\rm Ker} A \cap 
\ZZ^{N})
\end{eqnarray}
in it. Then the above system is 
naturally identified with 
the left $D(X)$-module 
\begin{equation}
M_{A, c}=
D(X) / \left(\sum_{1 \leq i \leq n} 
D(X) Z_{i,c} + \sum_{\mu \in {\rm Ker} A 
\cap \ZZ^{N}} D(X) \square_{\mu} \right). 
\end{equation} 
Let $\D_{X}$ be the sheaf 
of differential operators 
over the ``algebraic variety" $X$ and 
define a coherent $\D_{X}$-module by 
\begin{equation}
\M_{A, c}=
\D_{X} / \left(\sum_{1 \leq i \leq n} 
\D_{X} Z_{i,c} + \sum_{\mu \in {\rm Ker} A 
\cap \ZZ^{N} } \D_{X} \square_{\mu}\right). 
\end{equation}
Then $\M_{A, c}$ is the localization of 
the left $D(X)$-module $M_{A, c}$ 
(see \cite[Proposition 1.4.4 (ii)]{H-T-T} etc.). 
In \cite{A} Adolphson proved that $\M_{A, c}$ 
is holonomic. In fact, he proved the 
following more precise result. 

\begin{definition}\label{AND} 
(\cite[page 274]{A}, see also 
\cite{Oka} etc.) For $z \in X=\CC^A$ we 
say that the Laurent polynomial 
$h_z(x)= \sum_{j=1}^N z_jx^{a(j)}$ is 
non-degenerate if for any face $\Gamma$ of 
$\Delta$ not containing the origin 
$0 \in \RR^n$ we have 
\begin{equation}\label{E-111} 
\left\{ x \in T=(\CC^*)^n \ | \ h_z^{\Gamma}(x) = 
\frac{\partial h_z^{\Gamma}}{\partial x_1}(x)
= \cdots \cdots = 
\frac{\partial h_z^{\Gamma}}{\partial x_n}
(x)=0 \right\} = \emptyset, 
\end{equation}
where we set $h_z^{\Gamma}(x)=
\sum_{j:a(j) \in \Gamma} z_jx^{a(j)}$. 
\end{definition} 
Let $\Omega \subset X$ 
be the Zariski open subset 
of $X$ consisting of $z \in X= \CC^A$ 
such that the Laurent polynomial 
$h_z(x)= \sum_{j=1}^N z_jx^{a(j)}$ 
is non-degenerate. 
Then Adolphson's result 
\cite[Lemma 3.3]{A} asserts 
that the holonomic $\D_X$-module 
$\M_{A, c}$ is an integrable 
connection on $\Omega$. Namely 
on $\Omega \subset X$ the 
characteristic variety of $\M_{A, c}$ 
is contained in the zero section 
of the cotangent bundle $T^*\Omega$. 
Now let $X^{\an}$ (resp. $\Omega^{\an}$) 
be the underlying 
complex analytic manifold 
of $X$ (resp. $\Omega$) and consider 
the holomorphic solution complex 
${\rm Sol}_X(\M_{A, c}) \in \BDC (X^{\an})$ 
of $\M_{A, c}$ defined by 
\begin{equation}
{\rm Sol}_X (\M_{A, c})=\rhom_{\D_{X^{\an}}}
((\M_{A, c})^{\an}, \sho_{X^{\an}})  
\end{equation}
(see \cite{H-T-T} etc. for the details). 
Then the above Adolphson's result implies 
that ${\rm Sol}_X(\M_{A, c})$ is a local 
system on $\Omega^{\an}$. 
We call the sections of this local system 
\begin{equation}
H^0{\rm Sol}_X(\M_{A, c}
)_{\Omega^{\an}} = \shom_{\D_{X^{\an}}}
((\M_{A, c})^{\an}, \sho_{X^{\an}}
)_{\Omega^{\an}}  
\end{equation} 
confluent $A$-hypergeometric functions. 
Moreover Adolphson proved 
the following fundamental result. 
Let $\Vol_{\ZZ}(\Delta ) \in \ZZ$ be the 
normalized $n$-dimensional 
volume of $\Delta$ i.e. the $n!$ times of the usual 
one $\Vol (\Delta ) \in \QQ$ 
with respect to the lattice 
$\ZZ^n \subset \RR^n$. 

\begin{definition}\label{NRC} 
(Gelfand-Kapranov-Zelevinsky 
\cite[page 262]{G-K-Z-2}) 
For a face $\Gamma$ of $\Delta$ 
containing the origin $0 \in \RR^n$ we 
denote by $\Lin (\Gamma ) 
\simeq \CC^{\dim \Gamma} \subset \CC^n$ 
the $\CC$-linear span of $\Gamma$. 
Then we say that the parameter vector 
$c \in \CC^n$ is non-resonant if 
for any face $\Gamma$ of $\Delta$ 
of codimension $1$ such that 
$0 \in \Gamma$  
we have $c \notin \{ \ZZ^n+ \Lin (\Gamma ) \}$. 
\end{definition} 

\begin{theorem}\label{ADL} 
(Adolphson \cite[Corollary 5.20]{A}) 
Assume that the parameter vector 
$c \in \CC^n$ is non-resonant. Then 
the rank of the local system 
$\shom_{\D_{X^{\an}}}
((\M_{A, c})^{\an}, \sho_{X^{\an}}
)_{\Omega^{\an}}$ on 
$\Omega^{\an}$ is equal to $\Vol_{\ZZ}(\Delta ) \in \ZZ$. 
\end{theorem}
This is a generalization of 
the famous result of Gelfand-Kapranov-Zelevinsky 
in \cite{G-K-Z-1} to the confluent case. 

\medskip 
Now let us introduce our main result. 
Fix $1 \leq j_0 \leq N$ and let $\mathbb{L} \simeq 
\CC$ be a complex line in $X= \CC^N$ parallel to 
the $j_0$-th axis of $\CC^N$ and satisfying the 
condition 
\begin{equation}
\#  \{ \mathbb{L} \cap ( X \setminus \Omega ) \} 
< + \infty. 
\end{equation}
Our result will not depend on the choice of 
such $\mathbb{L} \simeq 
\CC$. For $R>0$ such that 
$\mathbb{L} \cap ( X \setminus \Omega ) 
\subset \{ z \in \mathbb{L} \simeq \CC \ | \ |z|<R \}$ 
we define a circle $C \subset \mathbb{L} \simeq 
\CC$ in $\mathbb{L} \simeq \CC$ by 
$C= \{ z \in \mathbb{L} \simeq 
\CC \ | \ |z|=R \}$. We denote 
the characteristic polynomial of the 
monodromy of the confluent $A$-hypergeometric 
functions along $C$ by $\lambda_{j_0}^{\infty}(t) 
\in \CC [t]$. Then $\lambda_{j_0}^{\infty}(t)$ 
does not depend on the choice of the radius 
$R>0$. Note also that if 
the parameter vector $c \in \CC^n$ is 
non-resonant the degree of  
$ \lambda_{j_0}^{\infty}(t)$ 
is $\Vol_{\ZZ}( \Delta )$. Our main aim here 
is to give an explicit formula for 
$ \lambda_{j_0}^{\infty}(t)$ which generalizes 
the one in \cite{T-1}. Let 
$\Delta_1, \ldots, \Delta_r \prec \Delta$ be the 
facets of the $n$-dimensional polytope 
$\Delta \subset \RR_v^n$ such that 
$a(j_0) \in \Delta_i$ and $0 \notin \Delta_i$. 
For $1 \leq i \leq r$ let 
$\Gamma_{i1}, \Gamma_{i2}, \ldots, 
\Gamma_{i m_i} \prec \Delta_i$ be the facets 
of $\Delta_i$ such that $a(j_0) \notin \Gamma_{ij}$. 
We set 
\begin{equation}
\widehat{\Gamma_{ij}} := 
{\rm conv} ( \{ 0 \} \cup \Gamma_{ij}), 
\qquad \widehat{\Delta_i} := 
{\rm conv} ( \{ 0 \} \cup \Delta_i). 
\end{equation}
Let $\rho_{ij} \in \ZZ^n \setminus \{ 0 \}$ 
be the primitive inner conormal vector of the 
facet $\widehat{\Gamma_{ij}} \prec \widehat{\Delta_i}$ 
of the polytope $\widehat{\Delta_i}$ and set 
\begin{equation}
h_{ij}:= \langle \rho_{ij}, a(j_0) \rangle >0. 
\end{equation}
We call $h_{ij}>0$ the lattice height of 
the point $a(j_0)$ from 
the facet $\widehat{\Gamma_{ij}}$. 
By this definition of $h_{ij}$ 
we have 
\begin{equation}
\Vol_{\ZZ} 
( \widehat{\Delta_i} )=
 \sum_{j=1}^{m_i}h_{ij} \cdot \Vol_{\ZZ} 
( \widehat{\Gamma_{ij}} ). 
\end{equation} 
Then our main theorem is 
as follows. 

\begin{theorem}\label{MTM} 
Assume that the parameter vector 
$c \in \CC^n$ is non-resonant. Then 
we have 
\begin{equation}
\begin{aligned}
\lambda_{j_0}^{\infty}(t) & = & 
\prod_{i=1}^r \prod_{j=1}^{m_i} 
\Bigl\{ 
t^{h_{ij}}- \exp (-2 \pi \sqrt{-1} 
\langle \rho_{ij}, c \rangle ) 
\Bigr\}^{\Vol_{\ZZ}(\widehat{\Gamma_{ij}})} 
\\
& & \times 
\Bigl(t-1 \Bigr)^{\Vol_{\ZZ}(\Delta)- 
\sum_{i=1}^r \Vol_{\ZZ}(\widehat{\Delta_i})}. 
\end{aligned}
\end{equation}
\end{theorem}

\noindent We will prove this theorem in 
Section \ref{sec:3}. Our proof is based 
on an explicit construction 
of a basis of the corresponding rapid decay 
homology group. It also 
enables us to compute the whole monodromy 
operator (including the nilpotent part) 
in small dimensions. Note that by the method 
of \cite{T-1} which cannot be applied 
to the confluent case we obtain only the 
semisimple part of the monodromy.

\section{A decomposition of a certain rapid decay 
homology group}\label{sec:2}

In this section, for the preparation of 
the proof of Theorem \ref{MTM} we 
show a decomposition of a certain rapid decay 
homology group associated to a 
special Laurent polynomial. 

\medskip 
First of all, we recall the definition of 
real oriented blow-ups in Hien 
\cite{H-1} and \cite{H-2} etc. Let $Z$ be a 
complex manifold of dimension $n$ and 
$D \subset Z$ a normal crossing divisor 
in it. Then for each point $q \in Z$ by 
taking a local coordinate $(x_1,\ldots, x_n)$ 
of $Z$ on a neighborhood of $q$ such that 
$q=(0,\ldots, 0)$ and $D= \{ x_1 \cdots x_k 
=0 \}$ for some $0 \leq k \leq n$ we obtain 
a morphism 
\begin{eqnarray}
([0, \varepsilon ) \times S^1)^k 
\times B(0; \varepsilon )^{n-k} 
& \longrightarrow & 
B(0; \varepsilon )^{k} 
\times B(0; \varepsilon )^{n-k}
\\
( \{ (r_i, e^{\sqrt{-1} \theta_i}) \}_{i=1}^k, 
x_{k+1}, \ldots, x_n) 
 & \longmapsto & 
( \{ r_i e^{\sqrt{-1} \theta_i} \}_{i=1}^k, 
x_{k+1}, \ldots, x_n) 
\end{eqnarray}
to $B(0; \e )^k \times B(0; \e )^{n-k} 
\subset Z$, where we set $B(0; \e )= 
\{ x \in \CC \ | \ |x|< \e \}$ 
for sufficiently small $\e >0$. By patching 
these morphisms 
together naturally, we can construct 
a real $(2n)$-dimensional manifold 
$\tl{Z}$ with boundary and a morphism 
$\pi : \tl{Z} \longrightarrow Z$ which 
induces an isomorphism $\tl{Z} \setminus 
\pi^{-1}(D) \simto Z \setminus D$. We call 
it the real oriented blow-up of $Z$ along $D$. 

\medskip 
Let $g(x)$ be a Laurent polynomial on 
$T_0=( \CC^*)^{n-1}_x$. Assume that 
$g(x)$ is non-degenerate with respect to 
its Newton polytope $NP(g) \subset \RR^{n-1}$ 
and $\dim NP(g)=n-1$. 
Namely we assume that for any face 
$\Gamma$ of $NP(g)$ the $\Gamma$-part 
$g^{\Gamma}$ of $g$ satisfies the condition 
\begin{equation} 
\left\{ x \in T_0=(\CC^*)^{n-1}  
\ | \ g^{\Gamma}(x) = 
\frac{\partial g^{\Gamma}}{\partial x_1}(x)
= \cdots \cdots = 
\frac{\partial g^{\Gamma}}{\partial x_{n-1}}
(x)=0 \right\} = \emptyset .  
\end{equation}
For a positive 
integer $h>0$ we define a meromorphic 
function $G(x,t)$ on $W=T_0 \times \CC_t$ by 
\begin{equation}
G(x,t)= \frac{1}{g(x)^{h-1}t^h}. 
\end{equation}
Here we consider such a very special function 
in order to modify our construction of 
the basis of the rapid decay homology 
group in \cite[Section 7]{E-T-2}. 
Since our results below are technical, 
for the first reading the reader can skip 
them and directly go to Section \ref{sec:3}. 
Set $D= \{ (x,t) \in W \ | \ g(x) \cdot t=0 \} 
\subset W$ and $W^{\circ} =W \setminus D$.   
Let $\pi : \tl{W} \longrightarrow 
W$ be the real oriented blow-up of $W$ 
along the normal crossing divisor $D \subset W$ and 
$\iota : W^{\circ} \hookrightarrow \tl{W}$ 
the inclusion map. We identify $D_0= \{ 
x \in T_0 \ | \ g(x) \not= 0 \} \subset T_0$ 
with an open subset of $D$ naturally and set 
$\tl{D_0} = \pi^{-1}(D_0)$, 
$P= \tl{D_0} \cap \overline{ \{ {\rm Re} G \geq 0 \} }$ 
and $Q= \tl{D_0} \setminus P$. Note that 
the set $Q \subset \tl{D_0} \subset \tl{W}$ 
consists of the rapid decay directions of 
the function $\exp (G)$ over $D_0 \subset D$. 
Now let us consider the local system 
\begin{equation}
\LL = \CC_{W^{\circ}} x_1^{\beta_1} \cdots 
x_{n-1}^{\beta_{n-1}} g(x)^{\beta_{n}} 
t^{\beta_{n}} 
\end{equation}
($\beta =( \beta_1, \ldots, \beta_n) \in \CC^n$) 
on $W^{\circ}$. By abuse of notation, 
for $p \in \ZZ$ we set 
\begin{equation}
H_p^{{\rm rd}}(W^{\circ})=H_p (W^{\circ} \cup Q, 
Q, \iota_* ( \LL ) ), 
\end{equation}
where $H_p (W^{\circ} \cup Q, 
Q, \iota_* ( \LL ) )$ 
stands for the $p$-th 
relative twisted homology group of the pair 
$(W^{\circ} \cup Q, Q)$ with coefficients 
in the rank-one local 
system $\iota_*( \LL )$ 
on $\tl{W}$ 
(see \cite{A-K}, \cite{Paj}, \cite{Pham} etc.). 
Our aim here is to decompose 
the vector space $V(p):=H_p^{{\rm rd}}(W^{\circ})$ 
into a direct sum with respect to some facets of 
$NP(g)$. For this purpose, we fix 
a lattice point $a \in NP(g) \cap \ZZ^{n-1}$ 
and impose a condition on the coefficient 
$z \in \CC$ of $x^a$ in $g(x)$ as follows. 
First we define a new 
Laurent polynomial $\tl{g}(x)$ 
whose constant term is zero by 
\begin{equation}
\tl{g}(x)=z-x^{-a}g(x). 
\end{equation}
Then obviously we have $\{ x \in T_0 \ | \ g(x)=0 \} 
= \{ x \in T_0 \ | \ \tl{g} (x)= z \}$. 
It is well-known that there exists $M \gg 0$ 
such that the restriction 
\begin{equation}
\tl{g}^{-1}( \CC \setminus B(0;M) ) \longrightarrow 
\CC \setminus B(0;M)
\end{equation}
of the map $\tl{g}: T_0 \longrightarrow \CC$ 
is a locally trivial fibration. 
Now we require the coefficient 
$z \in \CC$ of $x^a$ in $g(x)$ to satisfy 
the condition $|z| > M$. 
This condition on $z \in \CC$ in particular 
implies that the hypersurface 
$\tl{g}^{-1}(z)=g^{-1}(0) \subset T_0$ 
is smooth. 
We denote by $\Delta_0$ the convex 
hull of $NP(\tl{g}) \cup \{ 0\}$ in $\RR^{n-1}$. 
Then it is easy to see that we have 
$\Delta_0=NP(g)-a$. Let $\Sigma_0$ be the 
dual fan of $\Delta_0$ and 
$\Sigma$ its smooth subdivision 
(see \cite{Fulton}, \cite{Oda} etc.). 
We denote by $X_{\Sigma}$ the smooth toric 
variety associated to $\Sigma$. Then 
$X_{\Sigma}$ is a compactification of 
$T_0=( \CC^*)^{n-1}_x$. 
Let $D_1, \ldots, D_l \subset X_{\Sigma}$ 
be the $T_0$-divisors in $X_{\Sigma}$ 
and for $1 \leq i \leq l$ denote by 
$n_i \geq 0$ the order of the pole of 
the meromorphic function $\tl{g} (x)$ 
on $X_{\Sigma}$ along $D_i$. 
Recall that $X_{\Sigma} \setminus 
T_0= \cup_{i=1}^l D_i$ is a normal crossing 
divisor in $X_{\Sigma}$. By the non-degeneracy 
of $g$, the closure 
$\overline{ \tl{g}^{-1}(0) } \subset X_{\Sigma}$ 
of the hypersurface $\tl{g}^{-1}(0) \subset T_0$ 
in $X_{\Sigma}$ intersects 
any $T_0$-orbit in some $D_i$ 
such that $n_i>0$ transversally. 
More precisely, if 
\begin{equation}
q \in (D_{i_1} \cap \cdots \cap D_{i_k} 
\cap \overline{\tl{g}^{-1}(0)} ) \setminus 
( \cup_{i \notin \{ i_1, \ldots, i_k \}} D_i ) 
\end{equation}
for some $1 \leq i_1< i_2 < \cdots < i_k \leq l$ 
such that $1 \leq k < \dim X_{\Sigma} =n-1$ 
and $\# \{ 1 \leq j \leq k \ | 
\ n_{i_j} >0 \} \geq 1$, 
then there exists a local 
coordinate $y=(y_1,y_2, \ldots, y_{n-1})$ 
of $X_{\Sigma}$ on a neighborhood of $q$ 
such that $q=0$, $D_{i_j}= \{ y_j=0 \}$ 
($1 \leq j \leq k$), 
\begin{equation}
\tl{g} (y)= \frac{y_{n-1}}{y_1^{n_{i_1}} 
y_2^{n_{i_2}}  \cdots y_k^{n_{i_k}} } 
\end{equation}
and $\overline{ \tl{g}^{-1}(0)} = 
\{ y_{n-1}=0 \}$. 
By this explicit description of $\tl{g}$ 
we see that for any $t \in \CC$ 
the closure 
\begin{equation}\label{EXPL} 
\overline{ \tl{g}^{-1}(t) } 
= \{ y_{n-1}=t y_1^{n_{i_1}} \cdots y_k^{n_{i_k}} \} 
\subset X_{\Sigma} 
\end{equation}
of the hypersurface 
$\tl{g}^{-1}(t) \subset T_0$ in $X_{\Sigma}$ 
is smooth in a neighborhood of $q$ and we have 
\begin{equation}
D_{i_1} \cap \cdots \cap D_{i_k} 
\cap \overline{\tl{g}^{-1}(t)} = 
D_{i_1} \cap \cdots \cap D_{i_k} 
\cap \overline{\tl{g}^{-1}(0)} 
\end{equation}
(see also \cite[Section 3.5]{Zaharia} etc.). 
In particular, the hypersurface 
$\overline{ \tl{g}^{-1}(z) } 
= \overline{ g^{-1}(0) } \subset X_{\Sigma}$ for the above 
fixed constant $z \in \CC$ with 
the condition $|z| > M>0$ has this property and 
intersects $D_{i_1} \cap \cdots \cap D_{i_k}$ 
transversally. Moreover, by taking $|z|>M$ 
large enough, we may assume also that 
$\overline{ \tl{g}^{-1}(z) } 
= \overline{ g^{-1}(0) }$ intersects 
any $T_0$-orbit in $X_{\Sigma}$ transversally. 

\begin{lemma}\label{AL} (cf. A'Campo's lemma) 
Let $q$ be a point in $X_{\Sigma} \setminus 
T_0= \cup_{i=1}^l D_i$. Assume that for 
$1 \leq i_1< i_2 < \cdots < i_k \leq l$ 
such that $\# \{ 1 \leq j \leq k \ | 
\ n_{i_j} >0 \} \geq 2$ we have 
\begin{equation}
q \in (D_{i_1} \cap \cdots \cap D_{i_k}) \setminus 
( \cup_{i \notin \{ i_1, \ldots, i_k \}} D_i \cup 
\overline{\tl{g}^{-1}(0)} ). 
\end{equation}
Then there exists an open neighborhood 
$U$ of $q$ in $X_{\Sigma}$ such that 
$U \cap \tl{g}^{-1}(z)= U \cap g^{-1}(0) 
\not= \emptyset$ and 
\begin{equation}
H_p \Bigl( \{ (U \times \CC_t) \cap W^{\circ} \} 
\cup Q, Q, \iota_* ( \LL ) \Bigr) \simeq 0 \qquad 
(p \in \ZZ) 
\end{equation}
for generic $\beta =( \beta_1, \ldots, \beta_n) 
\in \CC^n$. 
\end{lemma}

\begin{proof} 
By using \cite[Lemma 6.1]{E-T-2} the 
proof proceeds completely 
similarly to that of A'Campo's lemma 
for monodromy zeta functions in 
\cite{AC} (see Oka 
\cite[Example (3.7)]{Oka} etc.). 
We omit the details. 
\end{proof} 
Let $\Gamma_1, \ldots, \Gamma_m$ be the 
facets of $\Delta_0$ such that 
$0 \notin \Gamma_j$. For $1 \leq j \leq m$ 
let $T_j \simeq ( \CC^*)^{n-2}$ be the 
$(n-2)$-dimensional $T_0$-orbit in 
$X_{\Sigma}$ associated to 
$\Gamma_j \prec \Delta_0$ and 
$\rho_j \in \ZZ^{n-1} \setminus \{ 0 \}$ 
the primitive outer conormal vector 
of the facet $\Gamma_j \prec \Delta_0$. 
We denote by $d_j >0$ the value of 
the linear function $\rho_j$ on 
$\Gamma_j$ and call it 
the lattice height of the origin 
$0 \in \Delta_0$ from $\Gamma_j$. 
Then the order of the pole of 
the meromorphic function $\tl{g} (x)$ 
along $T_j$ is equal to 
$d_j >0$. Moreover we have 
\begin{equation}
\Vol_{\ZZ} 
( \Delta_0 )=  
 \sum_{j=1}^{m} d_j \cdot \Vol_{\ZZ} 
( \Gamma_{j} ). 
\end{equation}
Note that for any $1 \leq j \leq m$ there 
exists unique $1 \leq i \leq l$ such that 
$T_j$ is an open subset of $D_i$ and we have 
$d_j=n_i$. By the above explicit 
descriptions of $\tl{g}$ 
and $\overline{ \tl{g}^{-1}(z) } 
= \overline{ g^{-1}(0) }$ (see \eqref{EXPL}), 
for any $1 \leq j \leq m$ 
there exists a tubular 
neighborhood $W_j$ of $T_j^{\circ}= 
T_j \setminus \overline{\tl{g}^{-1}(0)}$ 
in $X_{\Sigma} \setminus 
\overline{\tl{g}^{-1}(0)}$ such that 
$W_j \cap \tl{g}^{-1}(z)= W_j \cap g^{-1}(0) 
\not= \emptyset$ and $W_j \cap \tl{g}^{-1}(z)
=W_j \cap  g^{-1}(0)$ 
is a covering of $T_j^{\circ}$ of degree 
$d_j$. 
For $p \in \ZZ$ set 
\begin{equation}
V(p)_j:=
H_p \Bigl( \{ (W_j \times \CC_t) \cap W^{\circ} \} 
\cup Q, Q, \iota_* ( \LL ) \Bigr). 
\end{equation}
Then by Lemma \ref{AL} and Mayer-Vietoris 
exact sequences for relative twisted 
homology groups, for generic 
$\beta =( \beta_1, \ldots, \beta_n) 
\in \CC^n$ we obtain isomorphisms 
\begin{equation}
\oplus_{j=1}^m V(p)_j 
\simto V(p)= H_p^{{\rm rd}}( W^{\circ} ) \qquad 
(p \in \ZZ ). 
\end{equation}
Moreover as in the proof of 
\cite[Propositon 7.5]{E-T-2}, 
by using the twisted Morse theory on 
$T_j \simeq ( \CC^*)^{n-2}$ (with 
the help of our figure-8 construction 
of the rapid decay cycles 
in the two-dimensional case in 
\cite[Section 6]{E-T-2}) 
for any $1 \leq j \leq m$ 
we can show that 
$V(p)_j \simeq 0$ ($p \not= n$) and 
$\dim V(n)_j=h \cdot d_j \cdot \Vol_{\ZZ} 
( \Gamma_{j} )$ for generic 
$\beta =( \beta_1, \ldots, \beta_n) 
\in \CC^n$ and 
construct a basis of the vector space 
$V(n)_j$. More precisely, by our Morse theoretical 
construction of the basis, we obtain 
a filtration of the $\CC$-vector space 
$V(n)_j$ with $\Vol_{\ZZ}( \Gamma_j)$ 
subquotients of dimension $h \cdot d_j$.

\section{The proof of Theorem \ref{MTM}}\label{sec:3}

First of all, we briefly recall the results 
in \cite{E-T-2}. In \cite{E-T-2} 
we proved that if the parameter vector 
$c \in \CC^n$ is non-resonant 
any confluent $A$-hypergeometric function $u(z)$ 
has an integral representation of the form 
\begin{equation}\label{IRF} 
u(z)= \int_{\gamma^z} 
\exp (\sum_{j=1}^N z_j x^{a(j)}) 
x_1^{c_1-1} \cdots x_n^{c_n-1} 
dx_1 \wedge \cdots \wedge dx_n 
\end{equation}
on $\Omega^{\an} \subset X^{\an}$, 
where $\gamma =\{ \gamma^z \}$ is a family of 
real $n$-dimensional 
topological cycles $\gamma^z$ 
in the algebraic torus $T=(\CC^*)^n_x$ 
on which the function 
$\exp (h_z(x))= 
\exp (\sum_{j=1}^N z_j x^{a(j)})$ 
decays rapidly at infinity. 
More precisely $\gamma^z$ is an element of 
Hien's rapid decay homology group as follows. 
Assume that $z \in \Omega^{\an}$. 
Let $\Sigma_0$ be the dual fan of $\Delta$ in $\RR^n$ 
and $\Sigma$ its smooth subdivision 
(see \cite{Fulton}, \cite{Oda} etc.). 
Denote by $Z_{\Sigma}$ the smooth toric 
variety associated to the fan $\Sigma$ 
(see \cite{Fulton}, \cite{Oda} etc.). 
Then $Z_{\Sigma}$ is a smooth 
compactification of $T=(\CC^*)^n_x$ such 
that $Z_{\Sigma} \setminus T$ is 
a normal crossing divisor. Next by using the 
non-degeneracy of the Laurent polynomial 
$h_z(x)= \sum_{j=1}^N z_j x^{a(j)}$, as in 
\cite[Section 3]{M-T-3} and 
\cite[Section 3]{M-T-4} we construct a 
complex blow-up $Z:=\tl{Z_{\Sigma}}$ 
of $Z_{\Sigma}$ such that the meromorphic 
extension of $h_z$ to it has no point 
of indeterminacy. We say that an 
irreducible component of the normal 
crossing divisor $D=Z \setminus T$ is 
irrelevant if the meromorphic extension of 
$h_z$ to $Z$ has no pole along it. Denote by 
$D^{\prime} \subset D$ the union of the 
irrelevant irreducible components in $D$. 
Let $\pi : \tl{Z} 
\longrightarrow Z^{\an}$ be the real oriented 
blow-up of $Z^{\an}$ along $D^{\an}$ 
(see Section \ref{sec:2} and 
Hien \cite{H-1}, \cite{H-2} etc.) 
and set $\tl{D} =\pi^{-1}(D^{\an})$. 
By the natural open embedding 
$\iota : T^{\an} \hookrightarrow \tl{Z}$ 
we consider $T^{\an}$ as an open 
subset of $\tl{Z}$ and set 
$P_z= \tl{D} \cap 
\overline{ \{ x \in T^{\an}  \ | \ 
{\rm Re} h_z(x) \geq 0 \} }$ 
and $Q_z=\tl{D} \setminus 
\{ P_z \cup \pi^{-1}(D^{\prime})^{\an} \}$. Note 
that $Q_z$ is an open subset of $\tl{D}$ 
consisting of the rapid decay directions 
of the function $\exp (h_z(x))$. 
Finally let $\LL$ be the rank-one 
local system on $T^{\an}$ defined by 
$\LL =\CC_{T^{\an}} x_1^{c_1-1} 
\cdots x_n^{c_n-1}$. 
Then Hien's rapid decay homology 
groups that we need for the above integral 
representation are isomorphic to 
\begin{equation}
H_p^{\rd}(T)_z := 
H_p(T^{\an} 
 \cup Q_z, Q_z; \iota_*( \LL ) ) \qquad 
(p \in \ZZ ),  
\end{equation}
where $H_p(T^{\an} 
\cup Q_z, Q_z; \iota_*( \LL ) )$ 
stands for the $p$-th 
relative twisted homology group of the pair 
$(T^{\an} \cup Q_z, Q_z)$ with coefficients 
in the rank-one local 
system $\iota_*( \LL )$ 
on $\tl{Z}$ 
(see \cite{A-K}, \cite{Paj}, \cite{Pham} etc. 
and \cite[Proposition 3.4]{E-T-2}). 
In the proof of \cite[Theorem 4.5]{E-T-2} 
we proved that for any $z \in \Omega^{\an}$ 
we have $H_p^{\rd}(T)_z \simeq 0$ 
$(p \not= n)$ and 
the dimension of $H_n^{\rd}(T)_z$ is 
$\Vol_{\ZZ}( \Delta )$. Let $\CH_{n}^{\rd}$ 
be the local sytem of rank 
$\Vol_{\ZZ}( \Delta )$ on $\Omega^{\an}$ 
whose stalk at $z \in \Omega^{\an}$ 
is isomorphic to $H_n^{\rd}(T)_z$ 
(see also Hien-Roucairol \cite{H-R}). 
Namely sections of $\CH_{n}^{\rd}$ 
are continuous family 
of rapid decay $n$-cycles 
in $T^{\an}$ for the function 
$\exp (h_z(x))$. Then one of the 
main results of \cite{E-T-2} 
is as follows. 

\begin{theorem}\label{INR} 
( \cite[Theorem 4.5]{E-T-2} ) 
Assume that the parameter vector 
$c \in \CC^n$ is non-resonant. 
Then we have an isomorphism 
\begin{equation}\label{ETT} 
\CH_{n}^{\rd} \simeq 
\shom_{\D_{X^{\an}}}
(( \M_{A, c} )^{\an}, \sho_{X^{\an}}) 
\end{equation}
of local systems on $\Omega^{\an}$. 
Moreover this isomorphism is given by 
the integral 
\begin{equation}
\gamma \longmapsto \left\{ \Omega^{\an} \ni z 
\longmapsto \int_{\gamma^z} 
\exp (\sum_{j=1}^N z_j x^{a(j)}) 
x_1^{c_1-1} \cdots x_n^{c_n-1} 
dx_1 \wedge \cdots \wedge dx_n \right\}, 
\end{equation}
where for a continuous family $\gamma$ of rapid 
decay $n$-cycles in $\Omega^{\an} \times T^{\an}$ 
and $z \in \Omega^{\an}$ we 
denote by $\gamma^z \subset T^{\an}$ 
its restriction to $z \in \Omega^{\an}$. 
\end{theorem} 
By this theorem, the study of the monodromies 
at infinity of confluent $A$-hypergeometric 
functions is reduced to that of the rapid 
homology groups $H_n^{\rd}(T)_z$ 
$(z \in \Omega^{\an})$. In 
\cite[Sections 6 and 7]{E-T-2} we constructed 
a basis of $H_n^{\rd}(T)_z$ 
for generic $c \in \CC^n$. From now on, 
we shall prove Theorem \ref{MTM} 
by slightly modifying 
this construction with the help of 
the results in Section \ref{sec:2}. 

\medskip 
For the preparation, first let us show that it 
is enough to prove Theorem \ref{MTM} 
only for ``generic'' non-resonant $c \in \CC^n$. 
As in Adolphson \cite{A}, fix 
a point $z^{(0)} \in \Omega$ and consider 
the $\CC [z_1, \ldots, z_N]$-module 
\begin{equation}
\CC_{z^{(0)}} := \CC [z_1, \ldots, z_N]/ 
\sum_{j=1}^N \CC [z_1, \ldots, z_N] (z_j-z^{(0)}_j)
\simeq \CC. 
\end{equation}
Then for any non-resonant $c \in \CC^n$ 
the holonomic $\D_X$-module 
$\M_{A, c}$ is an integrable 
connection of rank $\Vol_{\ZZ} 
( \Delta )$ on a neighborhood of 
$z^{(0)} \in \Omega$ and we have an isomorphism 
\begin{equation}
\CC_{z^{(0)}} \otimes_{\CC [z_1, \ldots, z_N]} 
M_{A, c} \simeq \CC^{\Vol_{\ZZ}( \Delta )}. 
\end{equation}

\begin{lemma}\label{HDO} 
There exists a family of bases 
\begin{equation}
s_{1,c}, \ldots, 
s_{\Vol_{\ZZ}( \Delta ),c} \in 
\Gamma (U; \M_{A, c} ) 
\end{equation}
of the integral 
connection $\M_{A, c}$ on a neighborhood $U$ of 
$z^{(0)} \in \Omega$ parametrized by all 
non-resonant $c \in \CC^n$ such that their images 
\begin{equation}
\overline{s_{1,c}}, \ldots, 
\overline{s_{\Vol_{\ZZ}( \Delta ) ,c}} \in 
\CC_{z^{(0)}} \otimes_{\CC [z_1, \ldots, z_N]} 
M_{A, c} \simeq \CC^{\Vol_{\ZZ}( \Delta )} 
\end{equation}
generate the $\CC$-vector space 
$\CC_{z^{(0)}} \otimes_{\CC [z_1, \ldots, z_N]} 
M_{A, c}$ over $\CC$ and the connection matrices 
of $\M_{A, c}$ with respect to them depend 
holomorphically on non-resonant $c \in \CC^n$. 
\end{lemma}

\begin{proof}
We recall the constructions in Adolphson \cite{A}. 
Let $R$ be the $\CC$-algebra 
$\CC [x^{a(1)}, \ldots, x^{a(N)}]$ generated 
by the monomials $x^{a(1)}, \ldots, x^{a(N)}$ 
over $\CC$ and consider the free 
$\CC [z_1, \ldots, z_N]$-module 
$R [z_1, \ldots, z_N]= \CC [z_1, \ldots, z_N] 
\otimes_{\CC} R$. We endow $R [z_1, \ldots, z_N]$ 
with a structure of a left $D(X)$-module by 
defining the action of 
$\frac{\partial}{\partial z_j}$ by 
\begin{equation}
\frac{\partial}{\partial z_j} 
( \sum_a c_a(z)x^a)= \sum_a 
\frac{\partial c_a}{\partial z_j} (z)x^a 
+ \sum_a c_a(z)x^{a+a(j)} 
\end{equation}
$( \sum_a c_a(z)x^a \in 
R [z_1, \ldots, z_N])$. For $1 \leq i \leq n$ and 
$c \in \CC^n$ we define a differential operator 
$D_{i,c}$ on $R [z_1, \ldots, z_N]$ by 
\begin{equation}
D_{i,c}=x_i \frac{\partial}{\partial x_i} +c_i+ 
x_i \frac{\partial h_z}{\partial x_i} (x). 
\end{equation}
Then by \cite[Theorem 4.4]{A} there exists an 
isomorphism 
\begin{equation}
M_{A,c} \simeq R [z_1, \ldots, z_N]/ 
\sum_{i=1}^n D_{i,c} R [z_1, \ldots, z_N]
\end{equation}
of left $D(X)$-modules. 
This implies that for the differential operators 
\begin{equation}
D_{i,c, z^{(0)}}= 
x_i \frac{\partial}{\partial x_i} +c_i+ 
x_i \frac{\partial h_{z^{(0)}}}{\partial x_i} (x)
\end{equation}
on $R= \CC [x^{a(1)}, \ldots, x^{a(N)}]$ we 
have an isomorphism 
\begin{equation}
\CC_{z^{(0)}} \otimes_{\CC [z_1, \ldots, z_N]} 
M_{A,c} \simeq R / \sum_{i=1}^n D_{i,c, z^{(0)}} R. 
\end{equation}
Let $C( \Delta ) \subset \RR^n$ be the cone 
generated by $\Delta$ and consider the 
semigroup $C( \Delta ) \cap \ZZ^n$ in it. 
Let $\widehat{R}$ be the $\CC$-algebra 
generated by the monomials $x^a$ $(a \in 
C( \Delta ) \cap \ZZ^n)$. 
Then $\widehat{R}$ is isomorphic to the 
semigroup algebra $\CC [C( \Delta ) \cap \ZZ^n 
]$ of $C( \Delta ) \cap \ZZ^n$ over $\CC$. 
Note that $R$ is a 
subring of $\widehat{R}$. 
By \cite[Theorem 5.15]{A}, if $c \in \CC^n$ 
is non-resonant there exists an 
isomorphism 
\begin{equation}
R / \sum_{i=1}^n D_{i,c, z^{(0)}} R \simto 
\widehat{R} / \sum_{i=1}^n D_{i,c, z^{(0)}} 
\widehat{R}. 
\end{equation}
As explained in the proof of 
\cite[Corollary 5.11]{A}, the ring 
$\widehat{R}$ is Cohen-Macaulay for every $A$. 
So, as explained after \cite[Proposition 5.13]{A}, 
the reasoning of the proof of 
\cite[Theorem 5.4]{A} with $R$ replaced by 
$\widehat{R}$ is valid for every $A$. 
In particular, as in \cite[(5.7)]{A}, 
one can construct a finite-dimensional
subspace $G \subset \widehat{R}$ 
independent of non-resonant $c$, such that 
\begin{equation}
\widehat{R} = G \oplus 
( \sum_{i=1}^n D_{i,c, z^{(0)}} 
\widehat{R}). 
\end{equation}
Recall that in \cite[(5.7)]{A} the basis 
of $G$ is given by monomials  $x^a$ $(a \in 
C( \Delta ) \cap \ZZ^n)$ in $\widehat{R}$. 
In this way, we obtain a basis 
\begin{equation}
t_{1}, \ldots, 
t_{\Vol_{\ZZ}( \Delta )} \in 
\CC_{z^{(0)}} \otimes_{\CC [z_1, \ldots, z_N]} 
M_{A, c}
\end{equation}
of the $\CC$-vector space 
$\CC_{z^{(0)}} \otimes_{\CC [z_1, \ldots, z_N]} 
M_{A, c}$. By the proof of \cite[(5.7)]{A} 
it extends to the local one 
\begin{equation}
s_{1,c}, \ldots, 
s_{\Vol_{\ZZ}( \Delta ),c} \in 
\Gamma (U; \M_{A, c} ) 
\end{equation}
of the integral connection $\M_{A, c}$ on a 
neighborhood $U$ of $z^{(0)} \in \Omega$. 
Note also that by the proof of 
\cite[(5.7)]{A} the projection of a given element of 
$\widehat{R}$ to $G$ polynomially 
depends on $c$. Then 
by recalling the definitions of the 
differential operators 
$\frac{\partial}{\partial z_j}$ on $R$ and 
$\widehat{R}$ we see that 
the connection matrices 
of $\M_{A, c}$ with respect to 
its local basis constructed above depend 
holomorphically on non-resonant $c \in \CC^n$. 
This completes the proof. 
\end{proof}

\begin{corollary}
For any $1 \leq j_0 \leq N$ the characteristic 
polynomial $\lambda_{j_0}^{\infty}(t) \in 
\CC [t]$ depends holomorphically on 
non-resonant $c \in \CC^n$.
\end{corollary} 
\begin{proof}

Let $\M_{A,c}^*$ be the dual of 
the holonomic $\D_X$-module $\M_{A,c}$ (see 
\cite[Section 2.6]{H-T-T}). Recall that 
on $\Omega$ it is nothing but 
the dual connection of $\M_{A,c}$. 
Then by \cite[Proposition 4.2.1]{H-T-T} 
we have an isomorphism 
\begin{equation} 
 H^0{\rm Sol}_X(\M_{A, c}) \simeq 
H^{-N}DR_X( \M_{A,c}^*).  
\end{equation}
Moreover on $\Omega^{\an}$ the cohomology sheaf 
$H^{-N}DR_X( \M_{A,c}^*)$ is isomorphic to the 
local system consisting of the horizontal 
sections of the analytic connection 
$(\M_{A,c}^*)^{\an}$ associated to 
$\M_{A,c}^*$. 
Then by Lemma \ref{HDO} and the Cauchy-Kowalevsky 
theorem (with parameters), for any 
$z^{(0)} \in \Omega^{\an}$ there exists 
a basis 
\begin{equation}
u_{1}(z,c), \ldots, 
u_{\Vol_{\ZZ}( \Delta )}(z,c) \in 
H^0 {\rm Sol}_X (\M_{A, c})_{z^{(0)}} 
\end{equation}
of the $\CC$-vector space 
$H^0 {\rm Sol}_X (\M_{A, c})_{z^{(0)}} 
\simeq \CC^{\Vol_{\ZZ}( \Delta )}$ over 
$\CC$ which depends holomorphically on 
non-resonant $c \in \CC^n$. This implies 
that for any $1 \leq j_0 \leq N$ 
the monodromy matrix for 
$\lambda_{j_0}^{\infty}(t)$ 
has the holomorphic dependence on 
non-resonant $c \in \CC^n$. 
\end{proof}
By this corollary,  it 
suffices to prove Theorem \ref{MTM} 
only for ``generic'' non-resonant $c \in \CC^n$. 

\medskip 
Now we return to the proof of 
Theorem \ref{MTM}. First let us 
consider the case $n=2$. 
For simplicity, here we treat only the 
case where $r=1$ and $a(j_0) \in \Int (\Delta_1)$. 
The general case can be proved similarly. 
Let $h >0$ be the lattice height of 
the origin 
$0 \in \Delta \subset \RR^2$ from 
the $1$-dimensional facet 
$\Delta_1 \prec \Delta$ 
(for the definition see Sections 
\ref{sec:1} and  \ref{sec:2}). Then 
we have the equality $\Vol_{\ZZ} 
( \widehat{\Delta_1} )=h \cdot \Vol_{\ZZ}(\Delta_1)$. 
Let $d_1 >0$ (resp. $d_2 >0$) be the lattice 
height of the point $a(j_0) \in \Int (\Delta_1)$ 
from the vertex facet $\Gamma_{11}$ (resp. 
$\Gamma_{12}$) of the segment $\Delta_1$. 
Here we define the lattice 
heights $d_j$ in the affine span 
$\Aff ( \Delta_1) \simeq \RR^1$ 
of $\Delta_1$ in $\RR^2$. 
Note that we have $\Vol_{\ZZ}(\Delta_1)=d_1+d_2$ 
and 
\begin{equation}
h \cdot d_j=h_{1j} \cdot \Vol_{\ZZ} 
( \widehat{\Gamma_{1j}} ) \qquad 
(j=1,2). 
\end{equation}
By a suitable automorphism 
$\Psi : ( \RR^2, \ZZ^2 ) \simeq ( \RR^2, \ZZ^2 )$ 
of $( \RR^2, \ZZ^2 )$ we may assume that 
$\Delta_1 \subset \{ v=(v_1,v_2) \ | \ 
v_2=h \} \subset \RR^2$ and $\Gamma_{11} 
=(p,h)$, $\Gamma_{12}=(q,h)$ for some 
integers $p, q \in \ZZ$ such that 
$p-q= \Vol_{\ZZ}(\Delta_1)=d_1+d_2$. 
In this situation, let $\Sigma$ be a 
smooth subdivision of the dual fan 
of $\Delta$ in $\RR^2$. Let 
$\tau \in \Sigma$ be the $1$-dimensional 
cone which corresponds to the facet 
$\Delta_1 \subset \{ v_2=h \} \subset \RR^2$ 
of $\Delta$ and $\sigma_1 \in \Sigma$ 
(resp. $\sigma_2 \in \Sigma$) the unique 
$2$-dimensional cone containing $\tau$ 
which corresponds to the 
vertex $\Gamma_{11}$ (resp. 
$\Gamma_{12}$) of $\Delta$. Note that the 
primitive vectors on the two edges of 
$\sigma_1$ are 
\begin{equation} 
b_1= 
\left(  \begin{array}{c}
      -1 \\
      k 
    \end{array}  \right), \ \qquad \ 
b_2= 
\left(  \begin{array}{c}
      0 \\
      -1 
    \end{array}  \right) 
\end{equation} 
for some integer $k \in \ZZ$. Let 
$Z_{\Sigma}$ be the smooth toric variety 
associated to $\Sigma$ and 
$\CC^2( \sigma_1) 
\simeq \CC^2_y \subset Z_{\Sigma}$ 
its affine open subset associated to 
the $2$-dimensional cone $\sigma_1 
\in \Sigma$. On $\CC^2( \sigma_1) 
\simeq \CC^2_y$ the multi-valued function 
$x_1^{c_1-1}x_2^{c_2-1}$ can be written as 
\begin{equation}\label{R1} 
x_1^{c_1-1}x_2^{c_2-1}=
y_1^{ \langle b_1, c-e \rangle} \cdot 
y_2^{ \langle b_2, c-e \rangle }, 
\end{equation}
where we set 
\begin{equation}
e= 
\left(  \begin{array}{c}
      1 \\
      1 
    \end{array}  \right). 
\end{equation}
Moreover on $\CC^2( \sigma_1) 
\simeq \CC^2_y$ 
our Laurent polynomial 
$h_z(x)= \sum_{j=1}^N z_j x^{a(j)}$ 
can be written in the form 
\begin{equation}\label{R2} 
h_z(y)=y_1^{kh-p}y_2^{-h} \times 
 \tl{h_z}(y), 
\end{equation}
where $\tl{h_z}(y)$ is a 
polynomial i.e. its Newton polytope is 
contained in the first quadrant 
$\RR^2_+$ of $\RR^2$. Note also 
that the $T$-orbit $T_1 \simeq \CC^*$ 
in $Z_{\Sigma}$ which corresponds to 
$\tau \in \Sigma$ is 
$\{ y=(y_1,y_2) \ | \ y_1 \in \CC^*, 
y_2=0 \} \simeq \CC^*_{y_1} 
\subset \CC^2( \sigma_1 )$. Now 
let $g(y_1)$ be the restriction of 
the polynomial $\tl{h_z}(y)$ 
to the $T$-orbit $T_1 \simeq \CC^*_{y_1}$ 
in $Z_{\Sigma}$. It is easy to see that 
this polynomial $g(y_1)$ of $y_1$ contains 
the term $z_{j_0}y_1^{d_1}$ and its 
Newton polytope is the closed 
interval $[0,p-q]=[0,d_1+d_2]$ in $\RR^1$. 
Note that $g(y_1)$ is naturally identified 
with the $\Delta_1$-part $h_z^{\Delta_1}$ of $h_z$. 
By the non-degeneracy of $h_z$ we have 
$\#  \{ y_1 \in T_1 \ | \ g(y_1)=0 \} 
= \Vol_{\ZZ}(\Delta_1)=p-q= d_1+d_2$. Let 
$Z= \tl{Z_{\Sigma}} \longrightarrow 
Z_{\Sigma}$, $\pi : \tl{Z} \longrightarrow Z$, 
$Q_z \subset \tl{D}$, $\iota : T^{\an} 
\hookrightarrow \tl{Z}$ etc. be as before. 
We denote the strict transform of 
$T_1 \simeq \CC^* \subset Z_{\Sigma}$ 
in $Z= \tl{Z_{\Sigma}}$ 
also by $T_1$. Let $W$ be a sufficiently 
small tubular neighborhood of $T_1$ 
in $Z= \tl{Z_{\Sigma}}$ and set 
$W^{\circ} = W \cap T^{\an}$.

\begin{center}
\noindent\includegraphics[height=6cm]{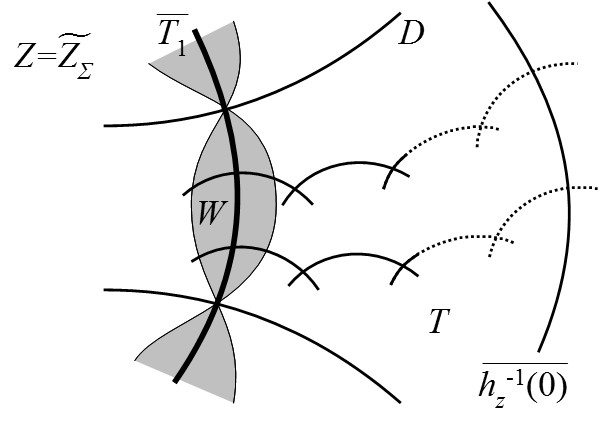}

Fig. 1: The situation in the blow-up $Z$
\end{center}

Then in \cite[Theorem 6.4]{E-T-2} 
we showed that if $c \in \CC^2$ is generic 
\begin{equation}
H_2^{\rd}(W^\circ)_z = 
H_2(W^\circ \cup Q_z, Q_z; \iota_*(\LL))
\end{equation}
is a linear subspace of $H_2^{\rd}(T)_z =
 H_2(T^{\an} \cup Q_z, Q_z; \iota_*(\LL))$.  Moreover 
in \cite[Proposition 6.2]{E-T-2} and 
the paragraph just after it, 
we constructed a natural basis of this 
$\Vol_{\ZZ}(\widehat{\Delta_1})$-dimensional vector space 
$H_2^{\rd}(W^\circ)_z$.  Recall that 
this basis was constructed 
by $d_1 + d_2$ figure-8s in $T_1 
\simeq \CC^*$ associated to the 
$d_1 + d_2$ points $\{y_1 \in T_1 
\mid g(y_1) = 0 \}$ in $T_1$.  
Let $C \subset  \mathbb{L}\simeq \CC$ be 
a sufficiently large circle in 
$\mathbb{L}\simeq \CC$ and for a 
point $z = (z_1, \ldots, z_N) \in C$ set 
$V = H_2^{\rd}(W^\circ)_z \subset H_2^{\rd}(T)_z$.  
Let $\Phi: V \simto V$ be the 
automorphism of $V$ induced by 
the rotation of the point $z_{j_0} \in C \subset 
\mathbb{L}\simeq \CC$ 
along $C$.  Then it suffices to show 
that the characteristic 
polynomial of $\Phi$ is equal to
\begin{equation}
\prod_{j=1}^2 \Bigl\{ 
t^{h_{1j}}- \exp (-2 \pi \sqrt{-1} 
\langle \rho_{1j}, c \rangle ) 
\Bigr\}^{\Vol_{\ZZ}(\widehat{\Gamma_{1j}})}.  
\end{equation}
Now let us set 
\begin{equation}
\tl{g}(y_1)=z_{j_0}-y_1^{-d_1}g(y_1). 
\end{equation}
Then the Newton polytope of 
the Laurent polynomial $\tl{g}(y_1)$ 
is the closed interval $[-d_1, d_2]$ in 
$\RR^1$ and for any $y_1 \in T_1 \simeq \CC^*$ 
we have an equivalence 
\begin{equation}
g(y_1)=0 \ \Longleftrightarrow \ 
\tl{g}(y_1)=z_{j_0}. 
\end{equation}
Since $|z_{j_0}|$ is sufficiently large, the set 
$\{ y_1 \in T_1 \ | \ g(y_1)=0 \} = 
\{ y_1 \in T_1 \ | \ \tl{g}(y_1)=z_{j_0} \}$ 
splits into two parts, i.e. the one 
consisting of the $d_1$ points 
$q_1, \ldots, q_{d_1} \in T_1$ in a 
neighborhood of the origin of $\CC_{y_1}$ 
and the other consisting of the remaining 
$d_2$ points $q_1^{\prime}, 
\ldots, q_{d_2}^{\prime} \in T_1$ 
at infinity. By this splitting, we can 
slightly modify the construction of the basis 
of $V = H_2^{\rd}(W^\circ)_z$ in 
\cite[Section 6]{E-T-2} so that we 
have a direct sum decomposition $V = V_1 \oplus V_2$ of $V$, 
where $V_1 \simeq \CC^{h \cdot d_1}$ 
(resp. $V_2 \simeq \CC^{h \cdot d_2}$)  
has a basis consisting of $h \cdot d_1$ 
(resp. $h \cdot d_2$) 
rapid decay 2-cycles over the $d_1$ (resp. $d_2$) 
figure-8s associated 
to the $d_1$ points $q_1,\ldots,q_{d_1} 
\in T_1 \simeq \CC^*$ 
(resp. the $d_2$ points 
$q'_1,\ldots,q'_{d_2} \in T_2 \simeq \CC^*$) as follows. 
Let us explain the construction of the 
vector space $V_1$. 
By homotopy we may assume that for 
some $0 < \varepsilon \ll 1$ we have 
\begin{equation}
q_i = \varepsilon \exp \left( 2\pi 
\sqrt{-1}  \frac{d_1-i+1}{d_1}  
\right) \in T_1 = \CC^*_{y_1}\quad (1\leq i \leq d_1).
\end{equation}
Note that by the rotation of the 
point $z_{j_0} \in C \subset  
\mathbb{L}\simeq \CC$ along $C$ 
the point $q_i\ (1 \leq i \leq d_1-1)$ 
(resp. $q_{d_1}$) is sent to $q_{i+1}$ 
(resp. $q_1$). Let $F_i (1\leq i \leq d_1-1)$ 
(resp. $F_{d_1}$) be a figure-8 
surrounding the two points 
$q_i$ and $q_{i+1}$ (resp. $q_{d_1}$ and $q_1$) 
(see \cite[Section 6]{E-T-2}). 
\medskip

\begin{center}
\noindent\includegraphics[height=6cm]{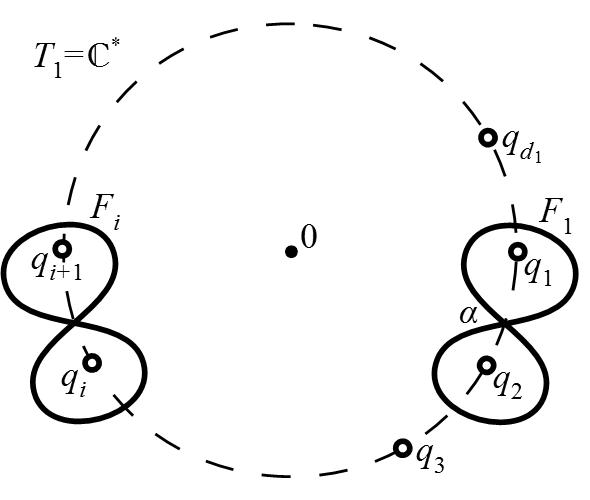}

Fig. 2: The construction of Figure-8s $F_i$
\end{center}

\medskip \noindent We take such figure-8s $F_1, F_2 , 
\ldots, F_{d_1} \subset T_1 = \CC^*_{y_1}$ 
so that by the rotation of $z_{j_0} \in C$ the $i-$th one 
$F_i (1 \leq i \leq d_i-1)$ 
(resp. the $d_1$-th one $F_{d_1}$)  is 
 sent to $F_{i+1}$ (resp. $F_1$). For the unique singular point 
$\alpha \in F_1$ of the figure-8 $F_1$ 
choose a normal slice 
$S_\alpha \simeq \CC_{y_2} \subset 
\CC^2(\sigma_1) \simeq \CC^2_y$ of 
$\overline{T}_1 \simeq \CC_{y_1} 
\subset \CC^2(\sigma_1) \simeq 
\CC^2_y$ at $\alpha\in \overline{T}_1$. Then 
$\pi^{-1}(S_\alpha) \subset \tl{Z}$ is 
isomorphic to the real oriented 
blow-up $\tl{\CC_{y_2}}$ of the complex 
plane $\CC_{y_2}$ along 
the origin $\{ 0 \} \subset \CC_{y_2}$ 
and the open subset 
$Q_z \cap \pi^{-1}(S_\alpha)$ of $\tl{D} 
\cap \pi^{-1}(S_\alpha) \simeq S^1$ 
consists of $h$ open intervals 
$I_1, I_2, \ldots, I_h \subset \tl{D}  
\cap \pi^{-1}(S_\alpha) \simeq S^1$.  
We may assume that 
$I_1, I_2, \ldots, I_h$ are arranged in 
the clockwise order.  
Let $\iota_0: S_\alpha \setminus \{\alpha\} \simeq \CC_{y_2}^* 
\hookrightarrow \tl{\CC_{y_2}}$ be the inclusion map.  
Then by \cite[Lemma 3.6(i)]{E-T-2} 
there exists a natural 
basis of the $h$-dimensional 
vector space 
\begin{equation} 
H_1(\CC_{y_2}^* 
\cup (I_1 \sqcup \cdots \sqcup I_h), 
I_1 \sqcup \cdots \sqcup 
I_h; (\iota_0)_*(\LL|_{S_\alpha 
\setminus\{\alpha\} })) 
\end{equation}
consisting of $h$ 
(twisted) 1-chains 
$\gamma_1, \gamma_2, \ldots, \gamma_h$ 
connecting two 
consecutive intervals among 
$I_1, \cdots, I_h$.  
\medskip

\begin{center}
\noindent\includegraphics[height=6cm]{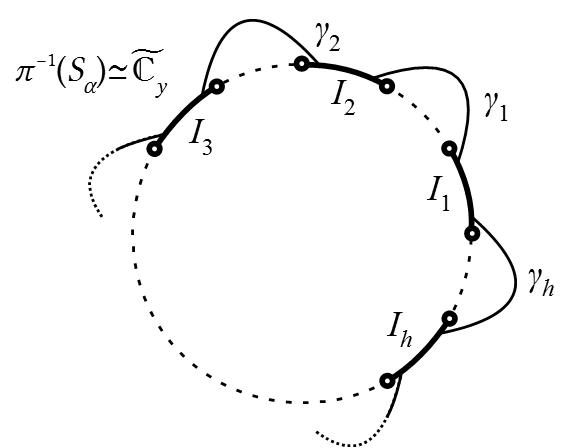}

Fig. 3: The twisted $1$-chains $\gamma_i$
\end{center}

\medskip \noindent We may assume 
that $\gamma_1, \gamma_2, \ldots, \gamma_h 
\subset \pi^{-1}(S_\alpha) \simeq \tl{\CC_{y_2}}$ 
are arranged in the clockwise order 
and the sections of 
the local system 
$(\iota_0)_*(\LL|_{S_\alpha \setminus\{\alpha\} })$ 
in the ``twisted'' 1-chains 
$\gamma_2, \ldots, \gamma_h$ are defined by 
the analytic continuations (in the clockwise direction) of 
that in the first one $\gamma_1$.  
As in the paragraph just after 
\cite[Proposition 6.2]{E-T-2}, 
by dragging $\gamma_1, \gamma_2, \ldots, \gamma_h$ 
over the figure-8 $F_1$ keeping their 
end points in $Q_z \subset \tl{D}$ 
we can naturally 
construct $h$ rapid decay 2-cycles $\delta_{11}, 
\delta_{12}, \ldots, \delta_{1h}$ over $F_1$ 
in $V = H_2^{\rd}(W^\circ)_z = H_2(W^\circ \cup Q_z, 
Q_z; \iota_*(\LL))$.  
By the rotation of the point 
$z_{j_0} \in C \subset 
\mathbb{L}\simeq \CC$ along $C$ (and the corresponding 
analytic continuations of the 
sections of $\iota_*(\LL)$) 
we obtain $h$ rapid decay 2-cycles $\delta_{21}, 
\delta_{22}, \ldots, \delta_{2h}$ over the figure-8 
$F_2$ in $V = H_2^{\rd}(W^\circ)_z$.  By repeating this 
construction, for any $1 \leq i \leq d_1$ we obtain 
\begin{equation}
\delta_{i1}, \delta_{i2},\ldots,
\delta_{ih} \in V = H_2^{\rd}(W^\circ)_z.
\end{equation}
As in \cite[Section 6]{E-T-2} we can show 
that if $c \in \CC^2$ is generic these $h \cdot d_1$ 
rapid decay 2-cycles in $V = 
H_2^{\rd}(W^\circ)_z \subset H_2^{\rd}(T)_z$ 
are linearly independent.  
We denote by $V_1 \simeq \CC^{h \cdot d_1}$ 
the linear subspace of $V$ spanned by them.  
Similarly by using 
the remaining $d_2$ points 
$q'_1,\ldots,q'_{d_2} \in T_1 
\simeq \CC^*_{y_1}$ at infinity, 
we obtain a linear subspace 
$V_2 \simeq \CC^{h \cdot d_2}$ 
of $V$ so that we have a direct 
sum decomposition $V = V_1 \oplus V_2$ of $V$.  
We can easily see that $V_1, V_2 \subset V$ are invariant 
by the automorphism $\Phi: 
V \simto V$ of $V$ induced by 
the rotation of the point $z_{j_0} \in C$.  
From now on, we 
will show that the characteristic polynomial of 
$\Phi_1 = \Phi|_{V_1}: V_1 \simto V_1$ is equal to 
\begin{equation}
\Bigl\{ 
t^{h_{11}}- \exp (-2 \pi \sqrt{-1} 
\langle \rho_{11}, c \rangle ) 
\Bigr\}^{\Vol_{\ZZ}(\widehat{\Gamma_{11}})}.  
\end{equation}
By our construction of the basis $\delta_{ij}$ 
$(1 \leq i \leq d_1, 1 \leq j \leq h)$ of $V_1 \subset V 
= H_2^{\rd}(W^\circ)_z$ for any $1 \leq i \leq d_1-1$ 
and $1 \leq j \leq h$ we have 
$\Phi_1(\delta_{ij}) = \delta_{i+1,j}$.  
For $m \in \ZZ$ we define two integers
$\{ m \}_h$ and $[m]_h$
such that $1 \leq [m]_h \leq h$ and
$m= h \cdot \{ m \}_h + [m]_h$. Then 
by \eqref{R1} and \eqref{R2} 
for any $1 \leq j \leq h$ we have
$\Phi_1 ( \delta_{d_1 j})=
\varepsilon_1 \cdot \varepsilon_2^{\{ j+p \}_h}
\cdot \delta_{1, [j+p]_h}$,
where we set $\varepsilon_1 = 
\exp(2 \pi \sqrt{-1} c_1)$ 
and $\varepsilon_2 = \exp(- 2 \pi \sqrt{-1} c_2)$.  
Define a square matrix $K \in M_h(\CC)$ of size $h$ by 
\begin{equation}
K = \varepsilon_1 \left( \begin{array}{ccc|c}
0&\cdots & 0 & \varepsilon_2 \\ \hline 
&&& 0 \\
&I_{h-1} && \vdots \\
&&& 0 \\
\end{array} \right)^p \in M_h(\CC),
\end{equation}
where $I_{h-1}$ is the $(h-1) \times 
(h-1)$ identity matrix.
Then by the basis $\delta_{ij}$ $(1 \leq i \leq d_1, 1 
\leq j \leq h)$ of $V_1$, the automorphism $\Phi_1: 
V_1 \simto V_1$ of $V_1$ is 
represented by the following 
square matrix $L \in M_{h\cdot d_1}(\CC)$ 
of size $h \cdot d_1$:
\begin{equation}
L = \left(\begin{array}{cccccc}
0&&&&& K\\
I_h &&&&&0\\
0 & I_h &&\ubigzero&& \vdots\\
\vdots & 0 &\ddots&&&\vdots\\
\vdots  & \vdots &&\ddots&&\vdots\\
0 & 0 & \cdots & 0 & I_h & 0
\end{array}\right) \in M_{h\cdot d_1}(\CC),
\end{equation}
where $I_h \in M_h( \CC )$ stands 
for the identity matrix of size $h$. 
It follows that the characteristic 
polynomial of the automorphism 
$\Phi_1 = \Phi |_{V_1} : V_1 \simto V_1$ 
of $V_1$ is equal to 
\begin{equation}
\prod_{\zeta^h= \exp (-2 \pi \sqrt{-1}c_2)}
\Bigl\{ 
t^{d_1}- \exp (2 \pi \sqrt{-1} c_1 ) \cdot \zeta^p 
\Bigr\}. 
\end{equation}
It is easy to see that it can be rewritten as 
\begin{equation}
\Bigl\{ 
t^{h_{11}}- \exp (-2 \pi \sqrt{-1} 
\langle \rho_{11}, c \rangle ) 
\Bigr\}^{\Vol_{\ZZ}(\widehat{\Gamma_{11}})}.  
\end{equation}
This completes the proof for the case $n=2$. 

\medskip 
Let us consider the general case $n \geq 2$. 
For simplicity, here we consider only the case 
where $r=1$ and $a(j_0) \in \Int (\Delta_1)$. The proof 
for the general case is similar. 
Let $\Sigma$, $Z_\Sigma$, 
$Z = \tl{Z_\Sigma} \to Z_\Sigma$, 
$\pi: \tl{Z} \to Z$, $Q_z \subset \tl{D}$, 
$\iota: T^{\an} \hookrightarrow \tl{Z}$ 
etc. be as before.  We denote by 
$T_1 \simeq (\CC^*)^{n-1}$ the 
$T$-orbit in $Z_\Sigma$ associated 
to the facet $\Delta_1 \prec \Delta$ 
of $\Delta$.  We also denote by 
$T_1 \simeq (\CC^*)^{n-1}$ its 
strict transform in $Z = 
\tl{Z_\Sigma}$.  Let $W$ be a 
sufficiently small tubular 
neighborhood of $T_1$ in 
$Z = \tl{Z_\Sigma}$, and set $W^\circ = W \cap T^{\an}$. 
Recall that in 
\cite[Theorem 7.6]{E-T-2} 
we showed that if $c \in \CC^n$ is generic 
\begin{equation}
H^{\rd}_n(W^\circ)_z = 
H_n(W^\circ \cup Q_z, Q_z; \iota_*(\LL))
\end{equation}
is a linear subspace of $H_n^{\rd}(T)_z$.  Let $C \subset 
\mathbb{L}\simeq \CC$ be a sufficiently large circle in  
$\mathbb{L}\simeq \CC$ and for 
a point $z \in C$ set $V = H_n^{\rd}(W^\circ)_z 
\subset H_n^{\rd}(T)_z$.  
Let $\Phi : V \simto V$ be the automorphism 
of $V$ induced by the rotation of the point 
$z_{j_0} \in C$. Recall that the dimension 
of the vector space $V$ is equal to
\begin{equation}
\Vol_{\ZZ} 
( \widehat{\Delta_1} )=h \cdot \Vol_{\ZZ}(\Delta_1) = 
 \sum_{j=1}^{m_1}h_{1j} \cdot \Vol_{\ZZ} 
( \widehat{\Gamma_{1j}} ),
\end{equation}
where $h > 0$ is the lattice height of the origin 
$0 \in \Delta \subset \RR^n$ from the 
facet $\Delta_1 \prec \Delta$ 
(for the definition see 
Sections \ref{sec:1} and  \ref{sec:2}). 
For $1 \leq j \leq m_1$ let $d_j>0$ be the 
lattice height of the point $a(j_0) \in 
\Delta_1$ from the facet $\Gamma_{1j} \prec \Delta_1$. 
Here we define the lattice 
heights $d_j$ in the affine span 
$\Aff ( \Delta_1) \simeq \RR^{n-1}$ 
of $\Delta_1$ in $\RR^n$ 
so that we have the equality 
$h_{1j} \cdot \Vol_{\ZZ} 
( \widehat{\Gamma_{1j}} )=h \cdot d_j \cdot 
\Vol_{\ZZ}( \Gamma_{1j} )$. 
Then we can slightly modify the 
construction of the basis of $V$ 
in \cite[Section 7]{E-T-2} 
with the help of our results in 
Section \ref{sec:2} to obtain a direct sum 
decomposition 
\begin{equation}
V= \oplus_{j=1}^{m_1} V_j
\end{equation}
of $V$ by some vector subspaces $V_j$, where 
$\dim V_j= h \cdot d_j \cdot 
\Vol_{\ZZ}( \Gamma_{1j} )$ and 
$V_j$ has a basis consisting of 
$h \cdot d_j \cdot 
\Vol_{\ZZ}( \Gamma_{1j} )$ 
rapid decay $n$ cycles associated to the facet 
$\Gamma_{1j}$ of $\Delta_1$. 
Note that by our Morse theoretical 
construction of the basis of $V_j$ 
we have a filtration of the $\CC$-vector space 
$V_j$ with $\Vol_{\ZZ}( \Gamma_{1j} )$ 
subquotients of dimension $h \cdot d_j$. 
Moreover as in the case 
$n=2$ we can show that 
the characteristic polynomial of  
the automorphism of each 
subquotient induced by $\Phi|_{V_j} : V_j 
\simto V_j$ is equal to 
\begin{equation}
\Bigl\{ 
t^{h_{1j}}- \exp (-2 \pi \sqrt{-1} 
\langle \rho_{1j}, c \rangle ) 
\Bigr\}^{l_j}, 
\end{equation}
where $l_j>0$ is the lattice height of 
the origin $0 \in \widehat{\Gamma_{1j}}$ 
from the facet $\Gamma_{1j} 
\prec \widehat{\Gamma_{1j}}$. 
Then the assertion immediately follows. 
This completes the proof. 

\begin{remark}
For $1 \leq j_1< j_2 < \cdots < j_k \leq N$ 
$(k>0)$ assume that no two of the $k$ 
points $a(j_1), \ldots, a(j_k)$ lie on the 
same facet of $\Delta$ not containing the 
origin $0 \in \RR^n$. Let $z^{(0)}=(z_1^{(0)}, 
\ldots, z_N^{(0)}) \in X= \CC^N$ be a point 
such that $z_j^{(0)} \not= 0$ if and only if 
$j \in \{ j_1, \ldots, j_k \}$. Let $\mathbb{L} 
\simeq \CC$ be a complex line in $X= \CC^N$ 
parallel to the one generated by the non-zero 
vector $z^{(0)} \not= 0$ such that 
\begin{equation}
\#  \{ \mathbb{L} \cap 
( X \setminus \Omega ) \} 
< + \infty 
\end{equation}
and $C \subset \mathbb{L} \simeq 
\CC$ a sufficiently large circle in it. 
Then by the above proof, also for the 
monodromy of the confluent 
$A$-hypergeometric functions along such 
$C$ we obtain a straightforward 
generalization of Theorem \ref{MTM}. 
\end{remark}

\end{document}